\documentclass[a4paper,12pt]{article}
\usepackage{}
\usepackage{mathrsfs}
\usepackage{amssymb}
\usepackage{amsmath}
\usepackage{amsmath,amssymb,amsthm,graphics, graphicx,latexsym,amsfonts}
\usepackage{fancyhdr}
\usepackage{color}
\title{{Flip-distance between $\alpha$-orientations of graphs embedded on plane and sphere}\thanks{Research supported by NSFC (No. 11471273 and 11561058.)}}
\author{{ Weijuan Zhang$^{1,2}$,
          Jianguo Qian$^{1}$\footnote{Corresponding author, email address: jgqian@xmu.edu.cn} ,
          Fuji Zhang$^{1}$ }\\
\small  $^{1}$School of Mathematical Sciences,  Xiamen University\\
\small  Xiamen, Fujian 361005, P.R.China \\
\small  $^{2}$School of Mathematical Sciences, Xinjiang Normal University\\
\small  Urumqi, Xinjiang 830054, P.R.China \\}
\date{}

\usepackage{indentfirst}
\begin{document}
\maketitle

\newtheorem{lem}{Lemma}[section]
\newtheorem{thm}[lem]{Theorem}
\newtheorem{prop}[lem]{Proposition}
\newtheorem{cor}[lem]{Corollary}
\newtheorem*{pf}{Proof}

\begin{abstract}  Felsner introduced a  cycle reversal, namely  the `flip' reversal, for $\alpha$-orientations (i.e., each vertex admits a  prescribed out-degree) of a graph $G$ embedded on the plane and further proved that  the set of all the $\alpha$-orientations of $G$ carries a distributive lattice  with respect to the flip reversals. In this paper, we give an explicit formula for the minimum number of flips needed to transform one $\alpha$-orientation into another for graphs embedded on the plane or sphere, respectively.  \\

\end{abstract}

\section {\large Introduction}
Research in graph orientation has a long history that reveals many interesting structural insights and applications. A classical example would be the one given by Robbins in 1939, which states that an undirected graph  has a strongly connected orientation  if and only if it is 2-edge connected. This result was then generalized by Nash-Williams to strongly $k$-edge connected orientations for any positive $k$ \cite{Nash}.  In the study of graph orientation, a particular concern is to orient a graph with certain degree-constraints on its vertices.
Frank \cite{Frank} established a characterization for the existence of those orientations in which the in-degree of each vertex has to lie within certain bounds.  In \cite{Hakimi}, Hakimi gave a characterization for a graph to have an orientation such that each vertex has a  prescribed out-degree. Such orientation is later called the $\alpha$-orientation  \cite{Felsner}.

Graph orientation has many interesting connections with certain combinatorial structures in graphs, such as the spanning trees \cite{Felsner}, bipartite perfect matchings (or more generally, bipartite $f$-factors) \cite{Kenyon,Lam,Propp}, Schnyder woods \cite{Fraysseix}, bipolar orientations \cite{Fraysseix2} and 2-orientations of quadrangulations, primal-dual orientations \cite{Disser}, transversal structures \cite{Fusy}  and $c$-orientations of the dual of plane graph \cite{Knauer,Propp}. Remarkably, all these structures can be encoded as the  $\alpha$-orientations \cite{Felsner2}, which has extensive applications, e.g.,  in drawing algorithms \cite{Barrera,Bonichon,Fusy}, and  enumeration and random sampling of graphs \cite{Creed,Fusy2}.

To deal with the relation among orientations, cycle reversal has been shown as a powerful method since it preserves the out-degree of each vertex and the connectivity of the orientations. Various types of cycle reversals were introduced subject to certain problem-specific requirements. An earlier example is  the cycle transformation introduced by Nash-Williams \cite{Nash} which says that any two $k$-connected orientations of a $2k$-edge connected graph can be transformed from each other by a sequence of cycle transformations or path transformations.

Cycle reversal for the orientations of plane graphs (graphs embedded on the plane) received particular attention. For example, Nakamoto \cite{Nakamoto}  considered the 3-cycle reversal to deal with the $(*)$-orientations in plane triangulation where each vertex on the outer facial cycle has out-degree 1 while each of the other vertices has out-degree 3. In \cite{ZhangF}, Zhang et al. introduced the Z-transformation to study the connection among perfect matchings of hexagonal systems  and later was extended to general plane bipartite graphs \cite{Zhang2}.

For orientations of general plane graphs, a natural considering  of cycle  reversal is to reverse a directed facial cycle. However, an orientation of a plane graph does not always have such a directed facial cycle, even if it has `many directed cycles'. In \cite{Felsner}, Felsner introduced a type of cycle reversal, namely  the `flip',  defined on the so-called essential cycles and proved that, any $\alpha$-orientation of a plane graph $G$ can transform into a particular $\alpha$-orientation by a flip sequence and further proved that the set of all $\alpha$-orientations of $G$ carries a distributive lattice  with respect to the flip reversals. In a strongly connected $\alpha$-orientation  of a plane graph, it is known that an essential cycle is exactly an inner facial cycle. In this sense, the notion `essential cycle' is a very nice generalization of facial cycle, which has been widely applied in the study of $\alpha$-orientations of plane graphs.

In this paper, we give a necessary-sufficient condition for that an $\alpha$-orientation of a plane graph can be transform into another by a flip sequence. In contrast to a plane graph, we will see in the last section that any two $\alpha$-orientations of a sphere graph (a graph embedded on the sphere) can always be transformed from each other by a flip sequence. Further, we give an explicit formula of  the `flip-distance'  between two $\alpha$-orientations, that is, the minimum number of flips needed to transform  one $\alpha$-orientation into another, for plane graphs and sphere graphs, respectively.  In our study, the `standard cycle system' introduced in the following section and the idea of `length function'  \cite{Hassin} and `$\alpha$-potential' \cite{Felsner} defined on the faces of the embedded graph play  important roles.

\section{Preliminaries}
For a graph $G$  we denote by $V(G)$ and $E(G)$ the vertex set and edge set of $G$, respectively. An {\it orientation} of  $G$ is an assignment of a direction to each edge of $G$. A {\it plane graph} and {\it sphere graph} are an embedding of a planar graph on the plane and sphere, respectively. Given  a plane graph $G$ with $n$ vertices $v_1,v_2,\cdots,v_n$ and an out-degree function $\alpha=(\alpha(v_1),\alpha (v_2),\cdots, \alpha (v_n))$ of  $G$, an orientation $D$  of $G$ is called an $\alpha$-{\it orientation} if  $d^+_D(v_i)=\alpha (v_i)$ for all $v_i\in V(G)$, where $d^+_D(v_i)$ is the out-degree of $v_i$ in $D$. We call an out-degree function $\alpha$ {\it feasible} for a graph $G$ if  $\alpha$-orientation of $G$ exists. The question whether an out-degree function $\alpha$  is feasible for $G$ can be translated to the construction of a maximal flow in a graph associated with $G$, and therefore can be solved  in polynomial time \cite{Felsner}.

An edge in a graph $G$ is called $\alpha$-{\it rigid} if it has the same direction in every $\alpha$-orientation of $G$. For a cycle $C$ of a plane graph $G$, the {\it interior cut} of $C$ is the edge cut consisting of  all the edges
connecting $C$ to an interior vertex of $C$. A simple cycle $C$ of a plane graph $G$ is called {\it essential}  \cite{Felsner} (with respect to $\alpha$) if $C$ is chord-free, all the edges in the interior cut of $C$ are rigid and there exists an $\alpha$-orientation such that $C$ is directed.

A directed cycle $C$ of an $\alpha$-orientation of a plane graph is called {\it counterclockwise} (resp., {\it clockwise}), or ccw (resp., cw) for simplicity, if the interior region of $C$  is
to the left (resp., right) of $C$.  A {\it flip} taking on an essential ccw cycle $C$ is the reorientation of $C$ from ccw to cw  \cite{Felsner}. The flip reversal induces a partial ordering relation `$\preccurlyeq$' on the set of all $\alpha$-orientations of a plane graph,  that is, two $\alpha$-orientations $D$ and $D'$ have the relation $D\preccurlyeq D'$ if $D$ can be transformed  from $D'$ by a sequence of flips. We denote by $d(D',D)$  the {\it flip-distance} from $D'$ to $D$ if $D\preccurlyeq D'$.

By the definition of rigid edge, we can see that each component obtained from an $\alpha$-orientation $D$ by deleting all the rigid edges is strongly connected \cite{Knauer}.  To simplify our discussion, in the following we  restrict our attention only to strongly connected orientations and therefore, the underlying graph $G$ is 2-edge connected. Further, if $G$ is not 2-connected then a strongly connected $\alpha$-orientation $D$ of $G$ is the edge-disjoint union of the restriction  of $D$ on the 2-connected components of $G$. For this reason, we only consider the case that $G$ is 2-connected and, therefore, any inner facial cycle $f$ (boundary of an inner face) in $G$ is simple, that is, each vertex on $f$ appears only once.

 As mentioned earlier, in a strongly connected $\alpha$-orientation of a plane graph, since every edge is not rigid, an essential cycle now is exactly a facial cycle. We notice that this is also the case for a sphere graph. Therefore, the notion `flip'  for the ccw facial cycles and the `flip-distance' for the $\alpha$-orientations of a sphere graph  can be naturally defined. Further, for a sphere graph $G$, only in the very special case when $G$ is a simple cycle $C$,  $C$ is the facial cycle of two faces of $G$ and thus, in any orientation of $G$,  $C$ is ccw  with respect to one face if and only if $C$ is cw with respect to the other.

For a plane or sphere graph $G$, we denote by ${\cal F}(G)$ the set of all the faces of $G$. In particular, if $G$ is a plane graph then we denote by  $f_{\rm out}(G)$ and ${\cal F}_{\rm inner}(G)$ the outer face and the set of all the inner faces of $G$, respectively. For a cycle $C$ in a plane graph $G$, we denote by ${\cal F}_{\rm int}(C)$  the set of the faces in the
interior region of $C$. For two edge-disjoint cycles $C_1$ and $C_2$ of a plane graph, if ${\cal F}_{\rm int}(C_2)\subset{\cal F}_{\rm int}(C_1)$ then, except some possible vertices in common with $C_1$, the cycle $C_2$ lies in the interior region of $C_1$. In this sense, we also say that $C_2$ is {\it contained in} $C_1$, or conversely, $C_1$ {\it contains} $C_2$.

The following result will be used in our forthcoming argument.

\begin{lem}\label{lem1} \cite{Felsner}  If $C$ is a ccw cycle in a strongly connected $\alpha$-orientation $D$ of a 2-connected plane graph, then there is a  flip sequence $F$ consisting of $|{\cal F}_{\rm int}(C)|$ flips which reverses $C$ from ccw into cw and keep the orientations of all other edges of $D$ invariant.  More specifically, the number of the flips in $F$ taking on each inner face $f$ equals 1 if $f\in {\cal F}_{\rm int}(C)$, and equals 0 if $f\notin {\cal F}_{\rm int}(C)$.
\end{lem}

 A set ${\cal C}$ of edge disjoint directed cycles in a directed plane graph or sphere graph is called a {\it standard cycle system} if  any two cycles in ${\cal C}$ are pairwise uncrossed. In the case of plane graph, we see that any two cycles $C_1$ and $C_2$ in a  standard cycle system satisfy either ${\cal F}_{\rm int}(C_1)\cap {\cal F}_{\rm int}(C_2)=\emptyset$, or ${\cal F}_{\rm int}(C_1)\subset {\cal F}_{\rm int}(C_2)$, or ${\cal F}_{\rm int}(C_2)\subset {\cal F}_{\rm int}(C_1)$. A directed graph $F$ (not necessarily connected) is called {\it oriented Eulerian} if each component of $F$ is directed Eulerian, that is, the out-degree of each vertex equals its in-degree. The following lemma might be a known result but we give its proof for the self-completeness.

\begin{lem}\label{lem3} Any oriented Eulerian plane graph or sphere graph $F$ can be  partitioned into the union of pairwise edge disjoint and uncrossed directed cycles, that is,  a standard cycle system.
\end{lem}
\begin{proof}  We apply induction on $|{\cal F}(F)|$, that is, the number of the faces in $F$.

Let $f$ be a directed  facial cycle of $F$ (the existence of $f$ is obvious since $F$ is oriented Eulerian). Then $F\setminus E(f)$ is still an  oriented Eulerian graph, where $F\setminus E(f)$ is the subgraph of $F$ obtained from $F$ by removing the edges on $f$. Moreover, $F\setminus E(f)$  has less faces than $F$ has. So by the induction hypothesis, $F\setminus E(f)$  can be partitioned into a standard cycle system ${\cal C}$. Notice that the cycles in ${\cal C}$ and the facial cycle $f$ are edge disjoint and uncrossed. This means that ${\cal C}\cup \{f\}$ is a standard cycle system of $F$, which completes our proof.
\end{proof}
 \section{Graphs embedded on the plane}

\begin{lem}\label{lem2}  Let $D$ be a strongly connected $\alpha$-orientation of a 2-connected plane graph $G$. Let $C^-$ be a simple ccw cycle in $D$ and let ${\cal C}^+=\{C_1^+,C_2^+,\cdots,$ $C^+_t\}$ be the set of edge disjoint, uncrossed and pairwise exclusive simple cw cycles contained in  $C^-$. Then there is a  flip sequence $F$ consisting of
\begin{equation}
\left|{\cal F}_{\rm int}(C^-)\setminus\bigcup_{i=1}^t{\cal F}_{\rm int}(C^+_i)\right|
\end{equation}
 flips which reverses all cycles  $C^-,C_1^+,C_2^+,\cdots,C^+_t$ and keep the orientations of all other edges of $D$ invariant. More specifically, for any face $f$, the number of flips in $F$  taking on  $f$  equals 1 if  $f\in{\cal F}_{\rm int}(C^-)\setminus\bigcup_{i=1}^t{\cal F}_{\rm int}(C^+_i)$, and equals 0 otherwise.
 \end{lem}
\begin{proof} We apply induction on the number $t$. The assertion follows directly by Lemma \ref{lem1} if $t=0$.

Since $D$ is strongly connected, there is a directed path $P_1$ from a vertex $u$ on $C^-$ to a vertex $v$ on $C^+_1$ (in the degenerated case when  $C^-$  and $C^+_1$ has a common vertex, we may have $u=v$) and a directed path $P_2$ from a vertex $v'$ on $C^+_1$ to a vertex $u'$ on $C^-$. Moreover, we may assume that $P_1$ and $P_2$  are shortest. This means that, $u$ (resp., $v$) is the only common vertex of $P_1$ and $C^-$ (resp., $C^+_1$) while  $u'$ (resp., $v'$) is the only common vertex of $P_2$ and $C^-$ (resp., $C^+_1$).  Let
$$P=P_1(u\rightarrow v)\cup C^+_1(v\rightarrow v')\cup P_2(v'\rightarrow u'),$$
 where for a directed path or cycle $W$ and two vertices $u$ and $v$ on $W$, $W(u\rightarrow v)$ is the section of $W$  from $u$ to $v$. Thus, $P$ is a directed path that visits $C^+_1$ exactly once.

 We notice that $P$ may visit more cycles  in ${\cal C}^+$ other than $C^+_1$, say $C^+_2,\cdots,$ $C^+_k$ and, without loss of generality, we assume that $P$ successively visits $C^+_1,C^+_2,$ $\cdots,C^+_k$, see Figure 1(a). Moreover, along with the directed path $P$,  we may assume that $P$ visits $C^+_i$ exactly once and that $v_i$ (resp., $v'_i$) is the first (resp., last) vertex on $C^+_i$ for each $i\in\{1,2,\cdots,k\}$. On the other hand, since $P$ is shortest, $u$ and $u'$ are the only common vertices of $P$ and $C^-$.
\begin{center}
\scalebox{0.4}{\includegraphics{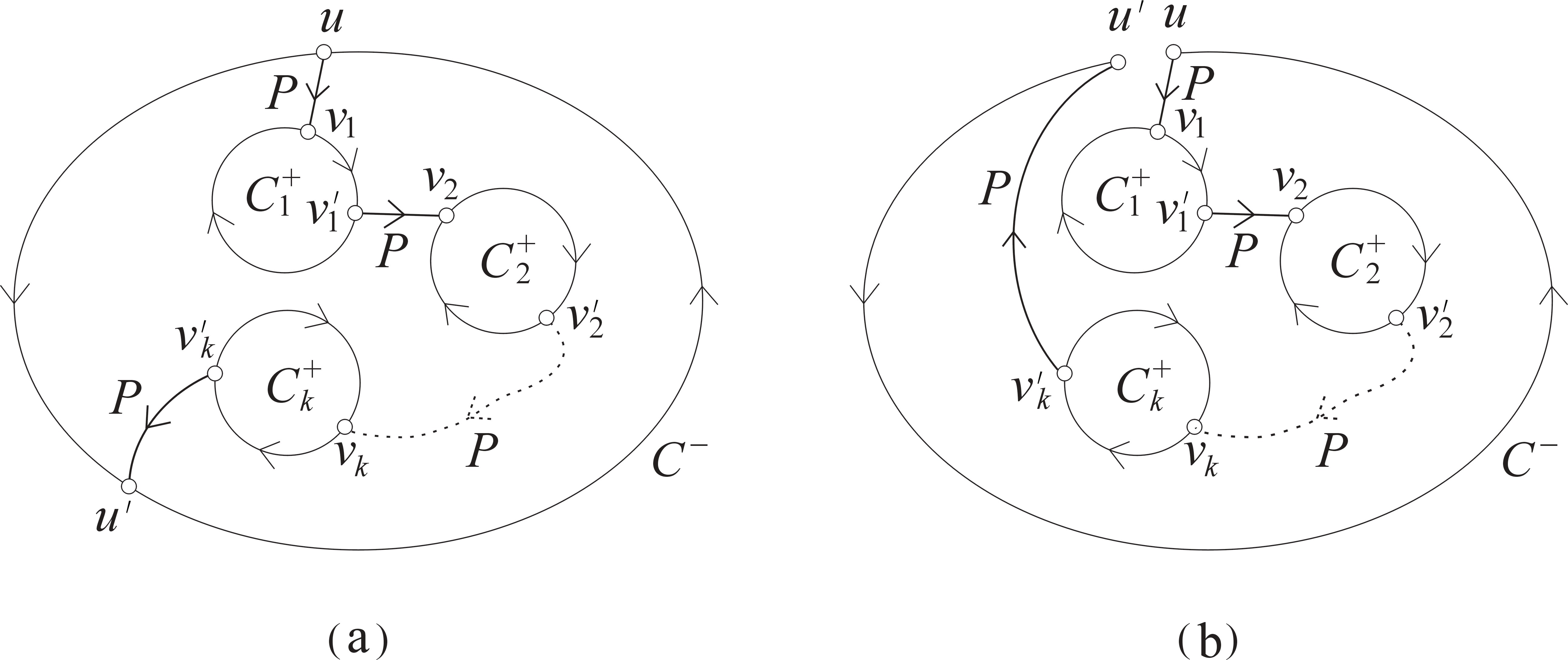}}\\
{\bf Figure 1.}
\end{center}

 Consider the directed cycle
$$C_1=P(u\rightarrow v_1)\cup C^+_1(v_1\rightarrow v'_1)\cup P(v'_1\rightarrow v_2)\cup C^+_2(v_2\rightarrow v'_2)\cup\cdots$$
 $$\cup\ C^+_k(v_k\rightarrow v'_k)\cup P(v'_k\rightarrow u')\cup C^-(u'\rightarrow u).$$

Clearly $C_1$ is a ccw cycle (possibly not simple) and contains less number of cw cycles in ${\cal C}^+$ than $C^-$ does. Moreover, since $G$ is 2-connected and $D$ is strongly connected, we can see that  the subgraph of $G$ induced by the vertices on $C_1$ and in the interior region of $C_1$ is still 2-connected and the sub-orientation of $D$ restricted on the sugraph is still strongly connected. Thus, if $C_1$ is simple then by the induction hypothesis, $C_1$ and those cw cycles $C^+_i$ contained in $C_1$ can be reversed by a flip sequence $F_1$ consisting of
 \begin{equation}
 \left|{\cal F}_{\rm int}(C_1)\setminus\bigcup_{{\cal F}_{\rm int}(C^+_i)\subset{\cal F}_{\rm int}(C_1)}{\cal F}_{\rm int}(C^+_i)\right|
 \end{equation}
flips. If $C_1$ is not simple, i.e., $u=u'$, then we can split $u=u'$ into two different vertices $u$ and $u'$ so that $C_1$ become to be a simple cycle $C^*_1$, see Figure 1(b). Notice that the `splitting' does not effect any flips since a flip only involves the edges of a facial cycle. Moreover, ${\cal F}_{\rm int}(C^*_1)={\cal F}_{\rm int}(C_1)$. The remaining discussion is analogous.

After $F_1$ being taken, the cycle
 $$C_2=P^{-1}(u'\rightarrow v'_k)\cup C^+_k(v'_k\rightarrow v_k)\cup P^{-1}(v_k\rightarrow v'_{k-1})\cup C^+_{k-1}(v'_{k-1}\rightarrow v_{k-1})\cup\cdots$$
 $$\cup C^+_1(v'_1\rightarrow v_1)\cup P^{-1}(v_1\rightarrow u)\cup C^-(u\rightarrow u')$$
is ccw, where $P^{-1}$ is the reversal of $P$. Again by the induction hypothesis, $C_2$  and those  $C^+_i$ contained in $C_2$ can be reversed by a flip sequence $F_2$ consisting of
\begin{equation}
\left|{\cal F}_{\rm int}(C_2)\setminus\bigcup_{{\cal F}_{\rm int}(C^+_i)\subset{\cal F}_{\rm int}(C_2)}{\cal F}_{\rm int}(C^+_i)\right|
\end{equation}
flips. By the construction of $C_1$ and $C_2$ we have
$${\cal F}_{\rm int}(C_1)\cup {\cal F}_{\rm int}(C_2)={\cal F}_{\rm int}(C^-)\setminus\bigcup_{i=1}^k{\cal F}_{\rm int}(C^+_i).$$
Moreover,
$$\{C^+_i: {{\cal F}_{\rm int}(C^+_i)\subset{\cal F}_{\rm int}(C_1)}\}\cup\{C^+_i: {{\cal F}_{\rm int}(C^+_i)\subset{\cal F}_{\rm int}(C_2)}\}\cup\bigcup_{i=1}^k\{C^+_i\}={\cal C}^+.$$

 Hence,  (1) follows directly from (2) and (3). This completes the proof.
\end{proof}

Let $f_1$ and $f_2$ be two adjacent faces and let $e$ be a common edge of $f_1$ and $f_2$. Along with the direction of $e$, if $f_1$ (resp., $f_2$) lies to the left  side of $e$ then we say that $f_1$ (resp., $f_2$)  is {\it left  of} $e$.

Following the idea of `length function' introduced in \cite{Hassin} and `$\alpha$-potential' introduced in \cite{Felsner}, we give the following definition.

\noindent{\bf Definition 3.1}\ For an oriented Eulerian subgraph $F$ of an $\alpha$-orientation $D$, define the $F$-{\it potential} $z_{F}(f)$ of $D$ as a function which assigns an integer to each face $f$ of $D$ according to the following rule:\\
1. $z_{F}(f_{\rm out})=0$;\\
2. if $f_1$ and $f_2$ are two faces sharing a common edge $e$ then
\[z_{F}(f_2)=
\begin{cases}
 \ z_{F}(f_1)+1,& \text{if\ $e\in F$\ and\ $f_2$\ is\ left\ of\ $e$},\\
 \ z_{F}(f_1)-1,& \text{if\ $e\in F$\ and\ $f_1$\ is\ left\ of\ $e$},\\
 \ z_{F}(f_1),& \text{otherwise}.\\
\end{cases}
\]

Let  ${\cal C}$ be a standard cycle system of $F$. By the definition of $z_{F}(f)$, for any face $f$, one can see that  the value of $z_{F}(f)$ is determined by those cycles in ${\cal C}$ that contain $f$ as an interior face. More specifically,
\begin{equation}
z_{F}(f)=c^--c^+,
\end{equation}
where $c^-$ and $c^+$ are the numbers of ccw and cw cycles in ${\cal C}$ which contain $f$ as an interior face, respectively, as illustrated in Figure 2. We notice that the value of $c^--c^+$ is a constant for any standard cycle system ${\cal C}$ of $F$. This means that $z_{F}(f)$ is determined uniquely by $F$ and therefore,  is well defined.
\begin{center}
\scalebox{0.6}{\includegraphics{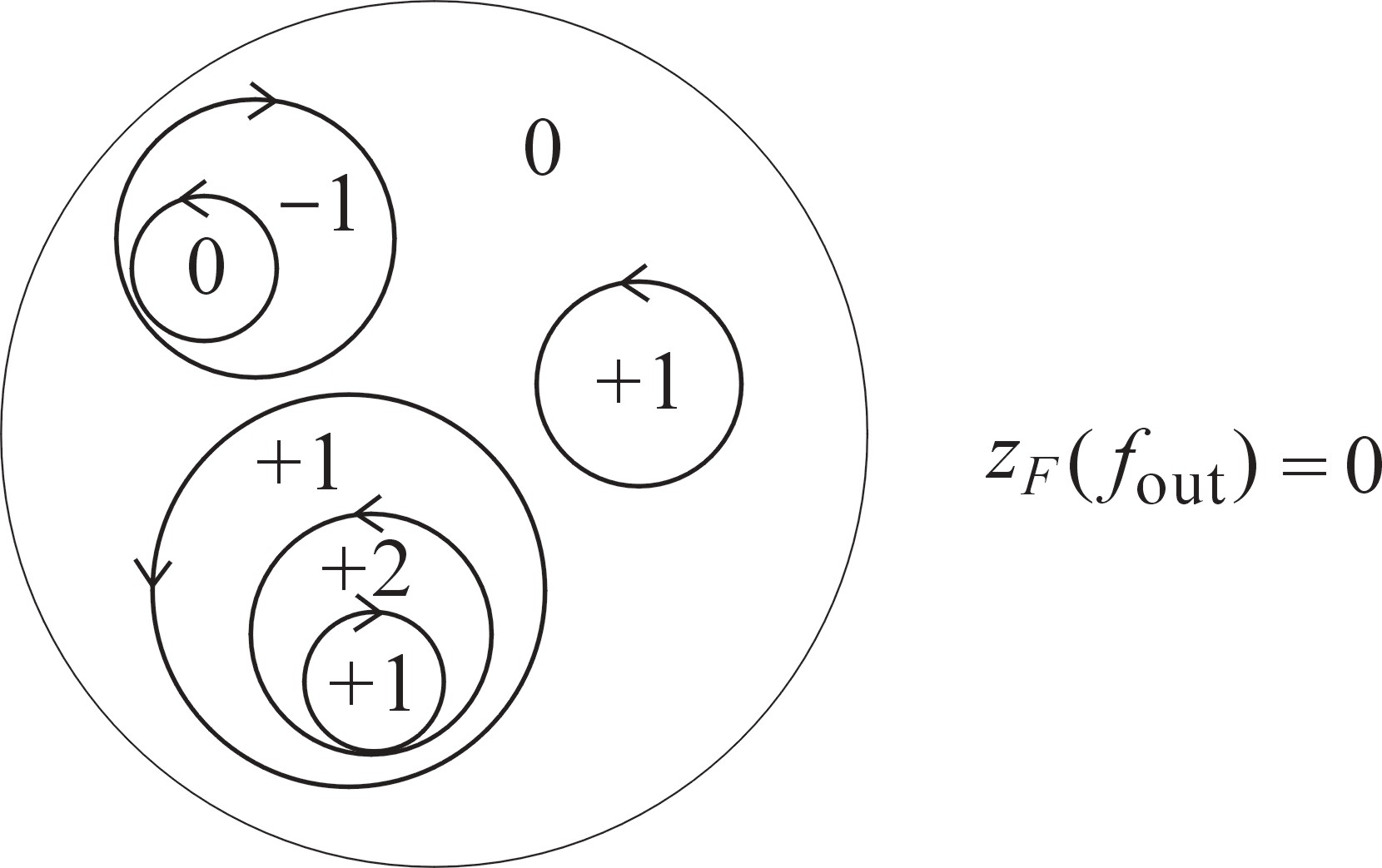}}\\
{\textbf{Figure\ 2}. The directed cycles in ${\cal C}$ are indicated in thick circles and the outer facial cycle is indicated in the thin circle.}
\end{center}

For two $\alpha$-orientations $D$ and $D'$, let $D'-D$ be obtained from $D'$ by removing those edges which has the same direction with $D$. We note that $D'-D$ is oriented Eulerian.

\begin{thm}\label{thm1} Let $D$ and $D'$ be two strongly connected $\alpha$-orientations of a 2-connected plane graph $G$.  Then $D\preccurlyeq D'$ if and only if, for any face $f\in{\cal F}(G)$,
\begin{equation}
 z_{D'-D}(f)\geq 0.
\end{equation}
Moreover, if $D\preccurlyeq D'$ then
\begin{equation}
d(D',D)=\sum\limits_{f\in{\cal F}(G)} z_{D'-D}(f).
\end{equation}
\end{thm}

\begin{proof}    Assume firstly that $D\preccurlyeq D'$. Let $F=\{f_1,f_2,\cdots,f_q\}$ be a flip sequence that transforms $D'$ into $D$.

For any face $f\in{\cal F}(G)$, consider the times $t(f)$ of flips in $F$  taking on $f$. Since the outer face $f_{\rm out}(G)$ is not involved in any flip, we have $t(f)=0$ if $f=f_{\rm out}(G)$. Now assume that $f'$ is a face which shares a common edge $e$ with $f$. If $e\in D'-D$, then the orientation of $e$ changes after $F$ being taken. Moreover, a flip to change the orientation of $e$ must be taken on the left face of $e$. This implies that $t(f)=t(f')+1$ if $f$ is left of $e$ or $t(f)=t(f')-1$  if $f'$ is left of $e$. If $e\notin D'-D$, then the orientation of $e$ does not change after $F$ being taken. This means that, if a flip in $F$ takes on $f$ then there is another flip that takes on $f'$ to keep the orientation of $e$ invariant and, hence,  $t(f)=t(f')$.  As a result, we have $t(f)= z_{D'-D}(f)$. Notice that $t(f)\geq 0$ for any $f\in {\cal F}(G)$. The necessity follows.

Conversely, assume that (5) holds. Let ${\cal C}$ be a standard cycle system of $D'-D$. If all the cycles in ${\cal C}$ are ccw then $D\preccurlyeq D'$ follows directly by Lemma \ref{lem1}. Now let $C^+_1$ be a maximal cw cycle in ${\cal C}$, that is, there is no  other cw cycle containing $C^+_1$ in ${\cal C}$. Then by (5), there is a ccw cycle $C^-\in{\cal C}$ satisfying  ${\cal F}_{\rm int}(C^+_1)\subset {\cal F}_{\rm int}(C^-)$ and, we choose $C^-$ to be minimal. Notice that $C^-$ may contain one more such cw cycles. We denote  by $C^+_1,C^+_2,\cdots,C^+_t$ all the maximal cw cycles of ${\cal C}$ which are contained in $C^-$.

Further, we notice that  $D'$ is strongly connected and $C^-,C^+_1,C^+_2,\cdots,C^+_t$ are in $D'$. So by Lemma \ref{lem2}, there is a flip sequence $F_1$ that reverses  $C^-$ and $C^+_i,i=1,2,\cdots,t$ and keep the orientations of all other edges of $D'$ invariant.
Let $(D'-D)^*={\cal C}\setminus\{C^-,C^+_1,C^+_2,\cdots,C^+_t\}$. Then by the choice of $C^-,C^+_1,$ $C^+_2,\cdots,C^+_t$, for any face $f$, we have \[z_{(D'-D)^*}(f)=
\begin{cases}
 \ z_{(D'-D)}(f)-1,& \text{if\ $f\in{\cal F}_{\rm int}(C^-)\setminus\bigcup_{i=1}^t{\cal F}_{\rm int}(C^+_i)$},\\
 \ z_{(D'-D)}(f),& \text{otherwise}.\\
\end{cases}
\]
 Moreover, by the maximality of $C^+_1,C^+_2,\cdots,C^+_t$, if
 $$f\in{\cal F}_{\rm int}(C^-)\setminus\bigcup_{i=1}^t{\cal F}_{\rm int}(C^+_i)$$
 then a cycle in ${\cal C}$ that contains $f$ as an interior face must be ccw. So by the definition of $z_{(D'-D)}(f)$, we have $z_{(D'-D)}(f)\geq 1$.

   The above discussion implies that $z_{(D'-D)^*}(f)\geq 0$ for any $f\in{\cal F}_{\rm int}(C^-)$. So by a simple induction on the number of the cycles in ${\cal C}$, there is a flip sequence $F_2$ that reverses  all cycles in $(D'-D)^*$.

Thus, $F_1\cup F_2$ reverse  all the cycles in ${\cal C}$, that is, transform $D'$ into $D$. The sufficiency follows.

 Finally, notice that $|F_1\cup F_2|$ equals the sum of the times $t(f)$ of flips in $F_1\cup F_2$ taking on each face $f\in{\cal F}(G)$. So by the proof of the necessity,
$$d(D',D)=|F_1\cup F_2|=\sum\limits_{f\in{\cal F}(G)}t(f)= \sum\limits_{f\in{\cal F}(G)}z_{D'-D}(f),$$
which completes our proof.
 \end{proof}

The following is an equivalent but more intuitive representation of (6).
\begin{cor} Let $D$ and $D'$ be two strongly connected $\alpha$-orientations of a 2-connected plane graph $G$ with $D\preccurlyeq D'$, and let $C_1^-,C_2^-,\cdots,C_s^-$ and $C_1^+,C_2^+,\cdots,C_t^+$ be all the ccw and cw cycles in  a standard cycle system of $D'-D$, respectively. Then
$$d(D',D)=\sum\limits_{i=1}^s|{\cal F}_{\rm int}(C_i^-)|-\sum\limits_{j=1}^t|{\cal F}_{\rm int}(C_j^+)|.$$
\end{cor}
\section{Graphs embedded on the sphere}
From the graph embedding point of view, a graph embedded on the sphere is essentially the same as  embedded on the plane. However, since a flip of an $\alpha$-orientation of a graph embedded on the sphere may be taken on any face of the graph while a flip for a graph embedded on the plane takes only on its inner face, the flip-distances for these two types of graph embedding are different. More specifically, we will see that, in contrast to plane graphs,  any two $\alpha$-orientations of a sphere graph can always be transformed from each other by a flip sequence.

Let $D$ and  $D'$ be two $\alpha$-orientations of a sphere graph $G$. Choose an arbitrary face $f\in {\cal F}(G)$, similar to the $(D'-D)$-potential $z_{D'-D}(f)$ for plane graph, we define
$$z_{D'-D}(f,g):{\cal F}(G)\rightarrow \mathbf{Z}$$
to be the function which assigns an integer to each face $g$ of $G$ according to the following rule, see Figure 3:\\
1.\ $z_{D'-D}(f,f)=0$;\\
2.\ if $g_1$ and $g_2$ are two faces sharing a common edge $e$ then
\[z_{D'-D}(f,g_2)=
\begin{cases}
 \ z_{D'-D}(f,g_1)+1,& \text{if\ $e\in D'-D$\ and\ $g_2$\ is\ left\ of\ $e$},\\
 \ z_{D'-D}(f,g_1)-1,& \text{if\ $e\in D'-D$\ and\ $g_1$\ is\ left\ of\ $e$},\\
 \ z_{D'-D}(f,g_1),& \text{otherwise}.\\
\end{cases}
\]

\begin{center}
\scalebox{0.33}{\includegraphics{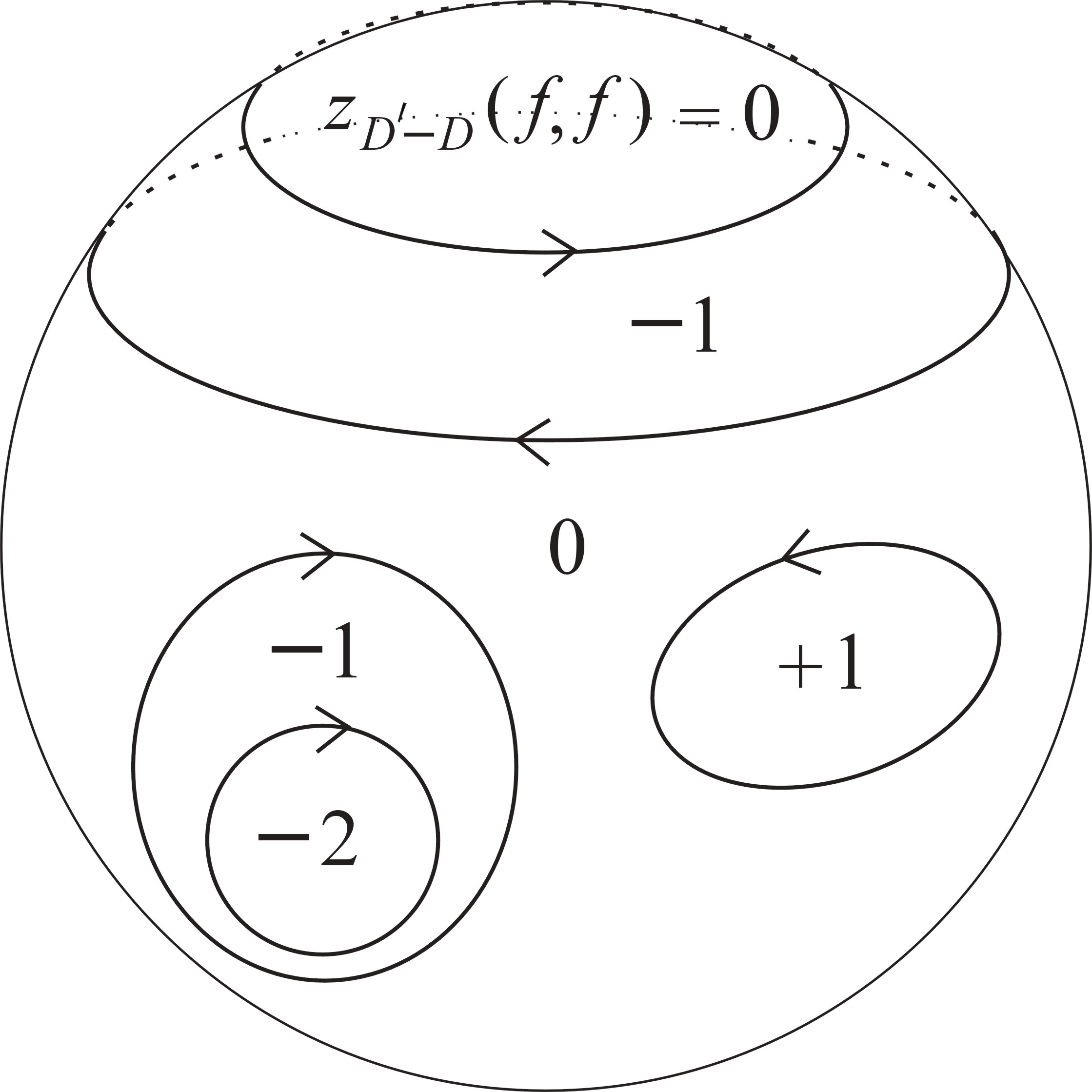}}\\
{\textbf{Figure\ 3}.}
\end{center}

For any $g\in {\cal F}(G)$, similar to the $F$-potential on plane graph,  we can see that $z_{D'-D}(f,g)$ is well defined.

\begin{thm}\label{sphere} Let $D$ and $D'$ be two strongly connected $\alpha$-orientations of a 2-connected sphere graph $G$ and let $f$ be an arbitrary face of $G$£®Then
\begin{equation*}
d(D',D)=\sum\limits_{g\in {\cal F}(G)}(z_{D'-D}(f,g)-z_{\min}(f)),
\end{equation*}
where
$$z_{\min}(f)=\min\{z_{D'-D}(f,g):g\in{\cal F}(G)\}).$$
\end{thm}
\begin{proof} Let $F$ be a flip sequence with minimum number of flips, say $q$ flips, that transform  $D'$ into $D$£®For any $g\in{\cal F}(G)$, let $z'(g)$ be the number of flips in $F$ that takes on $g$£®For any two adjacent faces $g$ and $g'$, from the definition of  $z_{D'-D}(f,g)$ we can easily see that
\begin{equation}
z'(g)-z'(g')=z_{D'-D}(f,g)-z_{D'-D}(f,g').
\end{equation}
That is, $z_{D'-D}(f,g_1)-z_{D'-D}(f,g_2)$ depends only on $D$ and $D'$, and does not depend on the choice of $f$£®In general, for any two faces $g$ and $g'$ (not necessarily adjacent), we can always find a face sequence
$$g=g_1,g_2,\cdots,g_p=g',$$
such that $g_i$ and $g_{i+1}$ are adjacent for each $i\in\{1,2,\cdots,p-1\}$. This means that (6) holds for any two faces $g$ and $g'$ of $G$. In particular, let $g'$ be a face such that
$$z_{D'-D}(f,g')=z_{\min}(f).$$
 Then we have
$$z'(g)-z_{D'-D}(f,g)=z'(g')-z_{D'-D}(f,g')=z'(g')-z_{\min}(f).$$

Therefore, the distance from $D'$ to $D$ satisfies
$$d(D',D)=q=\sum\limits_{g\in {\cal F}(G)}z'(g)=\sum\limits_{g\in {\cal F}(G)}(z_{D'-D}(f,g)+z'(g')-z_{\min}(f))$$
$$\geq\sum\limits_{g\in {\cal F}(G)}(z_{D'-D}(f,g)-z_{\min}(f)),$$
where the last inequality holds because $z'(g)\geq 0$ for any face $g\in {\cal F}(G)$£®

We now need only to find a flip sequence consisting of $\sum\limits_{g\in {\cal F}(G)}(z_{D'-D}(f,g)-z_{\min}(f))$ flips which transform  $D'$ into  $D$£®

Let ${\cal C}$ be a standard cycle system of  $D'-D$£®

For a directed cycle $C\in {\cal C}$, we can see that $C$ divides $G$ into two semi-sphere parts. Along with the orientation of $C$, let $G^+_C$  denote the subgraph of $G$ which lies in the right semi-sphere, including $C$ itself£®Similarly, let $G^-_C$ denote the subgraph of $G$ which lies in the left semi-sphere including $C$ itself£®In this way, $G^+_C$ and $G^-_C$ could be considered as two plane graphs. Moreover,  $C$ is cw in $G^+_C$ and ccw in $G^-_C$£®For convenience, we also use $G^+_C$ and  $G^-_C$ to  denote the `sub-orientation' of $D'$ restricted on $G^+_C$ and  $G^-_C$, respectively.

Choose two adjacent faces $g_0,g_1\in{\cal F}(G)$ such that $z_{D'-D}(f,g_0)=z_{\min}(f)$ and $z_{D'-D}(f,g_1)>z_{\min}(f)$. We note that, by the definition of $z_{D'-D}(f,g)$, such pair of $g_0$ and $g_1$ exists since ${\cal C}$ is not empty. Let $e$ be a common edge shared by $g_0$ and $g_1$. Then, $e$ must be on a cycle in ${\cal C}$, say  $C_0$ as shown in Figure 4(a). We draw $G^-_{C_0}$ and $G^+_{C_0}$ on the plane as shown in Figure 4(b) and (c), respectively.

In the plane graph $G^+_{C_0}$, with no loss of generality, let $C_1,C_2,$ $\cdots,C_t$ be all the maximal cycles of the cycle system ${\cal C}\setminus \{C_0\}$, that is, each $C_i$ is not contained in any other cycles of ${\cal C}\setminus \{C_0\}$. Since ${\cal C}$ is a standard cycle system of  $D'-D$, the interior region $B$ bounded by $C_0,C_1,C_2,\cdots,C_t$ does not
\begin{center}
\scalebox{0.62}{\includegraphics{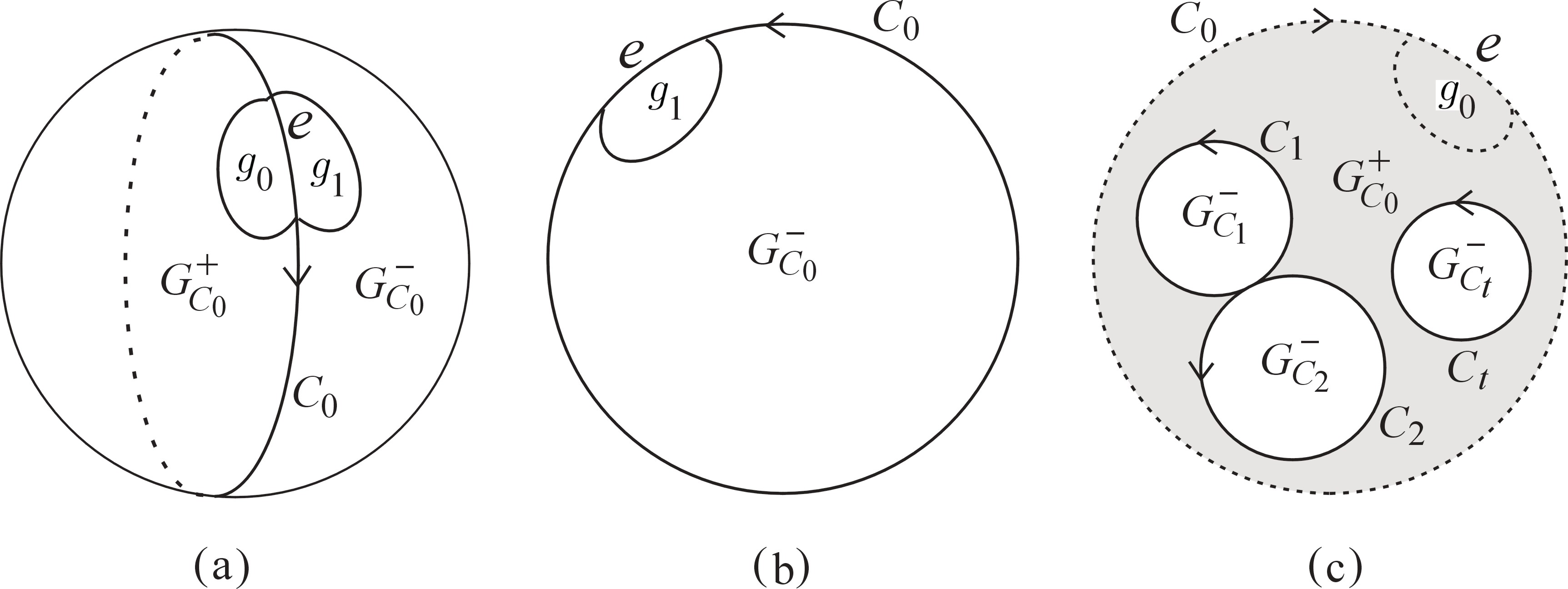}}\\
{{\bf Figure\ 4.} The interior region $B$  is indicated in grey.}
\end{center}
 contain any  cycle in ${\cal C}$. This means that, for any face
 $$g\in {\cal F}_{\rm inner}(G^+_{C_0})\setminus\bigcup_{i=1}^t{\cal F}_{\rm int}(C_i),$$
  we have
 $$z_{D'-D}(f,g)=z_{D'-D}(f,g_0)=z_{\min}(f).$$

Thus, again by the definition of $z_{D'-D}(f,g)$, the directed cycles $C_1,C_2,\cdots,$ $C_t$ are all ccw in the plane graph $G^+_{C_0}$.  Moreover, the plane graphs  $G^-_{C_i},i=0,1,2,\cdots,t$,  have the following properties:\\
1. a facial cycle in $G^-_{C_i}$ is ccw if and only if this facial cycle is ccw in $G$;\\
2. $G^-_{C_0}\cup G^-_{C_1}\cup \cdots \cup G^-_{C_t}$ contains all the cycles of ${\cal C}$;\\
3. $G^-_{C_i}$ is strongly connected.

For $i\in\{0,1,2,\cdots,t\}$, let $D_i$ and $D'_i$ be the orientations $D$ and $D'$ restricted on the plane graph $G^-_{C_i}$, respectively.   Then, for any $g\in {\cal F}_{\rm inner}(G^-_{C_i})$, by the definitions of $z_{D'_i-D_i}(g)$ (see Definition 2.1) and $z_{D'-D}(f,g)$,  we have
$$z_{D'_i-D_i}(g)= z_{D'-D}(f,g)-z_{\min}(f)$$
and therefore $z_{D'_i-D_i}(g)\geq 0$. Thus, by Theorem \ref{thm1}, all the cycles of ${\cal C}$ contained in $G^-_{C_i}$ can be reversed by
$$\sum\limits_{g\in{\cal F}_{\rm inner}(G^-_{C_i})} z_{D'_i-D_i}(g)=\sum\limits_{g\in{\cal F}_{\rm inner}(G^-_{C_i})}(z_{D'-D}(f,g)-z_{\min}(f))$$
flips. Hence, all cycles of ${\cal C}$ can be reversed by
$$\sum\limits_{i=0}^t\sum\limits_{g\in{\cal F}_{\rm inner}(G^-_{C_i})}(z_{D'-D}(f,g)-z_{\min}(f))=\sum\limits_{g\in {\cal F}(G)}(z_{D'-D}(f,g)-z_{\min}(f))$$
flips, where the last equality holds because $z_{D'-D}(f,g)=z_{\min}(f)$  for any
$$g\in{\cal F}(G)\setminus\bigcup_{i=0}^t {\cal F}_{\rm inner}(G^-_{C_i})= {\cal F}_{\rm inner}(G^+_{C_0})\setminus\bigcup_{i=1}^t{\cal F}_{\rm int}(C_i).$$
This completes our proof.
\end{proof}


\begin{thebibliography}{111}

\bibitem{Barrera}  C. Barrera-Cruz, Morphing planar triangulations, Ph.D Thesis,  The University of Waterloo,  Ontario, Canada, 2014.


\bibitem{Bonichon} N. Bonichon, S. Felsner,  M. Mosbah, Convex drawings of 3-connected planar graphs,
                   Algorithmica, 47 (2007), 399-420.

\bibitem{Creed} P.J. Creed, Counting and Sampling Problems on Eulerian Graphs, Ph.D Thesis, University of Edinburgh, 2010.

\bibitem{Disser}  Y. Disser, J. Matuschke, Degree-constrained orientations of embedded graphs, J. Comb. Optim., 31 (2016), 758-773.


\bibitem{Felsner} S. Felsner, Lattice structures from planar graphs,
                 Electron. J. Combin., 11 (2004), \#R15.

\bibitem{Felsner2}S. Felsner,  F. Zickfeld, On the number of planar orientations with prescribed degrees, Electron. J. Combin., 15 (2008), \#R77.

\bibitem{Frank} A. Frank, A.Gy\'{a}rf\'{a}s,  How to orient the edges of a graph,  Colloq. Math. Soc. J\'{a}nos  Bolyai, 18 (1976), 353-364.

\bibitem{Fraysseix} H. de Fraysseix, P.O. de Mendez, On topological aspects of orientation, Discrete Math., 229 (2001),  57-72.

\bibitem{Fraysseix2} H. de Fraysseix, P.O. de Mendez, P. Rosenstiehl, Bipolar orientations revisited, Discrete Appl. Math., 56 (1995), 157-179.

\bibitem{Fusy} E. Fusy, Transversal structures on triangulations: A combinatorial study and straight-line drawings,
               Discrete Math., 309 (2009), 1870-1894.

\bibitem{Fusy2} E. Fusy, D. Poulalhon, G. Schaeffer, {Dissections and trees, with applications to optimal mesh encoding and to random sampling}, In SODA'05: Proceedings of the 16th annual ACM-SIAM Symposium on Discrete Algorithms (2005), pp. 690-699.


\bibitem{Hakimi} S.L. Hakimi, On the degrees of the vertices of a directed graph, Journal of the Franklin Institute, 279 (4) (1965), 290-308.

\bibitem{Hassin} R. Hassin, Maximum flow in ($s,t$) planar networks, Information processing letters, 13(1981), 107-107

\bibitem{Kenyon} R. Kenyon, J. Propp, D.B. Wilson, Trees and matchings,
                 Electron. J. Combin., 7 (1) (2000), \#R25.

\bibitem{Knauer} K.B. Knauer, Partial orders on orientations via cycle flips,  Ph.D thesis, Technische Universit\"{a}t Berlin, Berlin, 2007.

\bibitem{Lam} P.C.B. Lam, H.P. Zhang, A distributive lattice on the set of perfect matchings of a plane bipartite graph, Order, 20 (2003), 13-29.



\bibitem{Nakamoto} A. Nakamoto, K. Ota, T. Tanuma, Three-cycle reversions in oriented planar triangulations, Yokohama Math. J., 44 (1997), 123-139.

\bibitem{Nash} C.St.J.A. Nash-Williams, On orientation, connectivity and odd-vertex pairing in finite graphs, Canad. J. Math., 12 (1960), 555-567.

\bibitem{Propp} J. Propp, Lattice structure for orientations of graphs,    arXiv:math/ 0209005 (2002).


\bibitem{ZhangF} F.J. Zhang, X.F. Guo, R.S. Chen,  $Z$-transformation graphs of perfect matchings of hexagonal systems,
                Discrete Math., 72 (1988), 405-415.

\bibitem{Zhang2} H.P.  Zhang, F.J. Zhang, {Plane elementary bipartite graphs}, Discrete Appl. Math., 105 (2000), 291-311.








\end{thebibliography}
\end{document}